\documentclass[12pt]{article}

\usepackage[dvipdfm, bookmarksopen=true]{hyperref}
\usepackage[centertags]{amsmath}
\usepackage{amsfonts}
\usepackage{amssymb}
\usepackage{amsthm}

\theoremstyle{plain}
\newtheorem{thm}{Theorem}[section]
\newtheorem{cor}[thm]{Corollary}
\newtheorem{lem}[thm]{Lemma}
\newtheorem{prop}[thm]{Proposition}

\theoremstyle{definition}
\newtheorem{defn}[thm]{Definition}
\newtheorem{exam}[thm]{Example}

\newtheorem{assum}[thm]{Assumption}
\theoremstyle{remark}
\numberwithin{equation}{section}

\newcommand{\beast}{\begin{eqnarray*}}
\newcommand{\eeast}{\end{eqnarray*}}
\title{The flipping puzzle on a graph\thanks{Research partially supported by the NSC grant 96-2628-M-009-015 of
Taiwan R.O.C..}}

\author{Hau-wen Huang\footnote{$~^\ddag$Department of Applied Mathematics National Chiao Tung University 1001 Ta Hsueh
Road Hsinchu, Taiwan 300, R.O.C..} \and Chih-wen Weng$~^\ddag$}

\date{May 19, 2009}

\begin{document}
\maketitle

\begin{abstract}
Let $S$ be a connected graph which contains an induced path of $n-1$ vertices, where $n$ is the
order of $S.$ We consider a puzzle on $S$. A configuration of the puzzle is simply an
$n$-dimensional column vector over $\{0, 1\}$ with coordinates of the vector indexed by the vertex
set $S$.  For each configuration $u$ with a coordinate $u_s=1$, there exists a move that sends $u$
to the new configuration which flips the entries of the coordinates adjacent to $s$ in $u.$ We
completely determine if one configuration can move to another in a sequence of finite steps.
\end{abstract}

\section{Introduction}\label{s1}

Let $S$ be a simple connected graph with vertex set $S=\{s_1, s_2, \ldots, s_n\}.$ By a {\it
flipping puzzle} on $S$, we mean a set of {\it configurations} of $S$  and a set of {\it moves} on
the configurations defined below. The configuration of the flipping puzzle is $S$, together with an
assignment of white or black  state  to each vertex of $S$. A move applied to a configuration $u$
in the puzzle is to select a vertex $s_i$ which has black state, and then flip the states of all
neighbors of $s_i$ in $u$. For convenience we use the set $F_2^n$ of column vectors over $F_2:=\{0,
1\}$, coordinates indexed by $S$, to denote the set of configurations of $S$. Precisely, for a
configuration $u\in F_2^n$, $u_{s_i}=1$ iff $u$ has black state in the vertex $s_i$.  Then for a
configuration $u$ with $u_{s_i}=1$ for some $s_i\in S$, we can apply a move to $u$ by changing $u$
into $u+A{\widetilde s}_i,$ where $A{\widetilde s}_i$ is the column indexed by $s_i$ in the
adjacency matrix $A$ of $S.$ A flipping puzzle is also called a {\it lit-only $\sigma$-game} in
\cite{xw:07}. The study of  flipping puzzles is related to the representation theory of Coxeter
groups \cite{hw:pre} and Lie algebras \cite{pb:00, pb:02, mkc:04, mkc:06, k:96}.
\bigskip

Two configurations  in the flipping puzzle on $S$ are said to be {\it equivalent} if one can be
obtained from the other by a sequence of selected moves. Let $\mathcal{P}$ denote the partition of
$F_2^n$ according to the above equivalent relation. A general question in solving the flipping
puzzle on $S$ is to realize that for a given pair  of configurations $u, v\in F_2^n$, whether $v$
can be obtained from $u$ by a sequence of selected moves or not. This can be done if $\mathcal{P}$
is completely determined.
\bigskip

In this paper we are mainly concerned about the class of graphs, each of which contains an induced
path on $\{s_1, s_2, \ldots, s_{n-1}\}$.  This class of graphs includes the simply-laced Dynkin
diagrams and simply-laced extended Dynkin diagrams  with exceptions $\widetilde{D}_n$ and
$\widetilde{E}_6.$ In each case of such graphs we determine $\mathcal{P}$. \bigskip

 For $u\in F_2^n$ let
$$w(u):=|\{ s_i\in S~|~u_{s_i}=1\}|$$ denote the Hamming weight of $u,$ and for an orbit
$O\in\mathcal{P},$ $$w(O):={\rm min}\{w(u)~|~u\in O\}$$ is called the {\it weight} of the orbit $O.$
The number $$M(S):={\rm max}\{w(O)~|~O\in \mathcal{P}\}$$ is called the {\it maximum-orbit-weight} of
the graph $S.$ A consequence of our result on $\mathcal{P}$ we find $M(S)\leq 2$ and we give a
necessary and sufficient condition for $M(S)=1.$ We also determine the cardinality of
$\mathcal{P}.$ A summary of our results is given in a table of Section~\ref{s7}. Besides these
results, a byproduct is Theorem~\ref{t3.9}.
\bigskip

If $S$ is a tree with $\ell$ leaves, X. Wang, Y. Wu \cite{xw:07} and H. Wu, G. J. Chang
\cite{wu:06} independently prove $M(S)\leq \lceil \ell/2\rceil.$ For each case of Dynkin diagrams
and extended Dynkin diagrams, $\mathcal{P}$ is completely determined by M. Chuah and C. Hu
\cite{mkc:04, mkc:06}. The study of flipping puzzles is related to a rich research subject called
"groups generated by transvections." We will provide this connection in Section~\ref{s8}.

\section{Matrices representing the puzzle}\label{s2}

Let $S$ be a simple connected graph with $n$ vertices. Let $F_2$ denote the $2$-element finite
field with addition identity $0$ and multiplication identity $1$, and let $F_2^n$ denote the set of
$n$-dimensional column vectors over $F_2$ indexed by $S$. We shall embed the graph $S$ in $F_2^n$
canonically. For $s\in S$, let $\widetilde{s}$ denote the characteristic vector of $s$ in $F_2^n$;
that is $\widetilde{s}=(0, 0, \ldots, 0, 1, 0, \ldots, 0)^t,$ where $1$ is in the position
corresponding to $s.$ The set $\{\widetilde{s}~|~s\in S\}$ is called the {\it standard basis} of
$F_2^n.$ In this setting, for $T\subseteq S$ the vector
$$\sum\limits_{s\in T} \widetilde{s}$$
represents the configuration with black states in $T$ in the flipping puzzle on $S$ as stated in
the introduction. We shall assign each move as an $n\times n$ matrix that acts on $F_2^n$ by left
multiplication. Let ${\rm Mat}_n(F_2)$ denote the set of $n\times n$ matrices over $F_2$ with rows
and columns indexed by $S.$

\begin{defn}\label{d2.1}
For $s\in S,$ we associate a matrix $\mathbf{s}\in {\rm Mat}_n(F_2)$, denoted by the bold type of
$s$, as
$$\mathbf{s}_{ab}=\left\{
\begin{array}{ll}
    1, & \hbox{if $a=b,$ or $b=s$ and $ab\in R$;} \\
0, & \hbox{else,} \\\end{array} \right.$$ where $a,b\in S$ and $R$ is the edge set of $S$. The
matrix $\mathbf{s}$ is called the {\it flipping move} associated with vertex $s.$
\end{defn}

It is easy to check that
 for $s, b\in S,$
$$\mathbf{s}\widetilde{b}=\left\{%
\begin{array}{ll}
  \widetilde{b}  , & \hbox{if $b\not= s$;} \\ %
\widetilde{b}+\sum\limits_{ab\in R}\widetilde{a} & \hbox{if $b=s$.} \\
\end{array}%
\right.
$$
Hence if a configuration $u\in F_2^n$ with $u_s=1$ then $\mathbf{s}u$ is the new configuration
after the move to select the vertex $s$. Note that if $u_s=0$, we have $\mathbf{s}u=u$, so we can
view the action of $\mathbf{s}$ on $u$ as a {\it feigning move} on $u$ which is not originally
defined as a move in the flipping puzzle. Note that $\mathbf{s}$ is an involution and hence is
invertible for $s\in S.$

\begin{defn}\label{d2.4}
Let $\mathbf{W}$ denote the subgroup of  ${\rm GL}_n(F_2)$ generated by the set
$\{\mathbf{s}~|~s\in S\}$ of flipping moves.  $\mathbf{W}$ is called the {\it flipping group} of
$S$.
\end{defn}

The flipping groups of simply-laced Dynkin diagrams are studied in \cite{hw:pre}. The flipping
group of the line graph of a tree with $n$ vertices is isomorphic to the symmetric group $S_n$ on
$n$ elements if $n\geq 3$ \cite{w:08}. However, we do not need the information of the flipping
group $\mathbf{W}$ of $S$ in this paper.

\section{The sets $\Pi,$ $\Pi_0$ and $\Pi_1$}\label{s3}

For the remaining of the paper, the following assumption is assumed.

\begin{assum}\label{ass1}
let $S$ be a simple connected graph with $n$ vertices $s_1,$ $s_2,$ $\ldots,$ $s_n,$ and  suppose
that the sequence  $s_1, s_2, \ldots, s_{n-1}$ is an induced  path, among them, $s_{j_1}, s_{j_2},
\ldots, s_{j_m}$   the neighbors of $s_n,$ where $1\leq j_1<j_2<\cdots<j_m\leq n-1.$ See Figure 1.
\end{assum}

\setlength{\unitlength}{1mm}
\begingroup\makeatletter\ifx\SetFigFont\undefined
% extract first six characters in \fmtname
\def\x#1#2#3#4#5#6#7\relax{\def\x{#1#2#3#4#5#6}}%
\expandafter\x\fmtname xxxxxx\relax \def\y{splain}% \ifx\x\y   % LaTeX or
%SliTeX?
\gdef\SetFigFont#1#2#3{%
  \ifnum #1<17\tiny\else \ifnum #1<20\small\else
  \ifnum #1<24\normalsize\else \ifnum #1<29\large\else
  \ifnum #1<34\Large\else \ifnum #1<41\LARGE\else
     \huge\fi\fi\fi\fi\fi\fi
  \csname #3\endcsname}%
\else \gdef\SetFigFont#1#2#3{\begingroup
  \count@#1\relax \ifnum 25<\count@\count@25\fi
  \def\x{\endgroup\@setsize\SetFigFont{#2pt}}%
  \expandafter\x
    \csname \romannumeral\the\count@ pt\expandafter\endcsname
    \csname @\romannumeral\the\count@ pt\endcsname
  \csname #3\endcsname}%
\fi\endgroup
\begin{picture}(100, 110)

%\put(0,25){\footnotesize{$E_8$}}% E_8
\put(30,25){\circle{1.5}} \put(38,25){\circle{1.5}} \put(54,58){\circle{1.5}}
\put(46,25){\circle{1.5}} \put(54,25){\circle{1.5}} \put(62,25){\circle{1.5}}
\put(70,25){\circle{1.5}} \put(78,25){\circle{1.5}} \put(30.5,25.5){\line( 1, 0){7}}
\put(71,25.5){\line( 1, 0){6.5}} \put(30,22){\scriptsize{$s_{n-1}$}}
\put(38,22){\scriptsize{$s_{n-2}$}} \put(46,22){\scriptsize{$s_{j_m}$}}
\put(54,22){\scriptsize{$s_{j_2}$}} \put(62,22){\scriptsize{$s_{j_1}$}}
\put(54,60){\scriptsize{$s_n$}} \put(70,22){\scriptsize{$s_2$}} \put(78,22){\scriptsize{$s_1$}}
\put(39,24){$\cdots$}\put(55,24){$\cdots$}\put(47,24){$\cdots$}\put(63,24){$\cdots$}
\put(54,26){\line( 0, 1){31}} \put(46,26){\line(1, 4){8}}\put(62,26){\line(-1, 4){8}} \put(48,
30){$\cdots$}
 \put(33,
10){{\bf Figure 1:} The graph $S$.}
\end{picture}

\bigskip
\bigskip

In the remaining of this paper, we always assume $n\geq 2$ and  set
\begin{equation}\label{e4.1}
\overline{1}=\widetilde{s}_1,~\overline{i+1}=\mathbf{s_i}\mathbf{s_{i-1}}\cdots
\mathbf{s_1}\overline{1}\quad{\rm for}~ 1\leq i\leq n-1.
\end{equation}
Set
\begin{eqnarray}
\Pi_{~}&=&\{\overline{1}, \overline{2}, \ldots, \overline{n}\},\label{ed1}\\
\Pi_0&=&\{\overline{i}\in\Pi~|~<\overline{i}, \widetilde{s}_n>=0\}\label{ed2},\\
\Pi_1&=&\Pi-\Pi_0,\label{ed3}
\end{eqnarray}
where $<~,~>$ is the dot product of vectors. From (\ref{e4.1}) and the construction,
\begin{eqnarray}
\Pi_0&=&\{\overline{i}~|~ \overline{i}=\widetilde{s}_{i-1}+\widetilde{s}_i,~1\leq i\leq n-1{\rm~or~}\overline{i}=\widetilde{s}_{n-1}\},\label{e3.2}\\
\Pi_1&=&\{\overline{i}~|~
\overline{i}=\widetilde{s}_{i-1}+\widetilde{s}_i+\widetilde{s}_n,~1\leq
i\leq
n-1{\rm~or~}\overline{i}=\widetilde{s}_{n-1}+\widetilde{s}_n\},\label{e3.3}
\end{eqnarray}
where $\widetilde{s}_0=0.$ Note that $1\leq |\Pi_0|, |\Pi_1| \leq n-1$ and $|\Pi_0|+|\Pi_1|=n.$
Precisely,
\begin{eqnarray}
\Pi_0 &=&\{\overline{i}\in\Pi~|~i\in (0, j_1]\cup (j_2, j_3] \cup \cdots \cup (j_{2k}, j_{2k+1}]\}\label{epr1}\\
\Pi_1 &=&\{\overline{i}\in\Pi~|~i\in (j_1, j_2]\cup (j_3, j_4] \cup \cdots \cup (j_{2k-1}, j_{2k}]\}
\end{eqnarray}
where $k= \lceil\frac{m}{2}\rceil$, $j_t:=n$ if $t>m$ and $(a, b]=\{x~|~x\in \mathbb{Z}, a<x\leq
b\}.$ In particular we have the following proposition.

\begin{prop}\label{p3.2}
$$
|\Pi_1|=
      \sum^{\lceil\frac{m}{2}\rceil}_{k=1}j_{2k}-j_{2k-1}.
$$ \hfill $\Box$
\end{prop}

From (\ref{e3.2}), (\ref{e3.3}), we immediately have the following lemma.

\begin{lem}\label{l3.2} For $1\leq i\leq n-1,$
$$\overline{1}+\overline{2}+\cdots+\overline{i}=\left\{
                                                  \begin{array}{ll}
                                                    \widetilde{s}_i+ \widetilde{s}_n, & \hbox{if $|[\overline{i}]\cap \Pi_1|$ is odd;} \\
                                                    \widetilde{s}_i, & \hbox{if $|[\overline{i}]\cap \Pi_1|$ is even,}
                                                  \end{array}
                                                \right.$$
and
$$\overline{1}+\overline{2}+\cdots+\overline{n}=\left\{
                                                  \begin{array}{ll}
                                                    \widetilde{s}_n, & \hbox{if $|\Pi_1|$ is odd;} \\
                                                    0, & \hbox{if $|\Pi_1|$ is even,}
                                                  \end{array}
                                                \right.$$
where $[\overline{i}]:=\{\overline{1}, \overline{2}, \ldots, \overline{i}\}.$ \hfill $\Box$
\end{lem}

From Lemma~\ref{l3.2} and (\ref{epr1}) we have the following lemma.

\begin{lem}\label{l4.1}
~~~~~~$\displaystyle\sum\limits_{\overline{i}\in \Pi_0}
\overline{i}=\sum\limits_{k=1}^m\widetilde{s}_{j_k}.$\hfill $\Box$
\end{lem}

From (\ref{e4.1}) we have the following lemma.

\begin{lem}\label{l4.2}
$\mathbf{s_i}\overline{i}=\overline{i+1},$ $\mathbf{s_i}\overline{i+1}=\overline{i}$ and
$\mathbf{s_i}$ fixes other vectors in $\Pi-\{\overline{i},\overline{i+1}\}$ for $1\leq i\leq n-1.$
\hfill $\Box$
\end{lem}

From~Lemma~\ref{l4.2}, $\mathbf{s_i}$ acts on $\Pi$ as the transposition $(\overline{i},
\overline{i+1})$ in the symmetric group $S_n$ of $\Pi$ for $1\leq i\leq n-1.$ Let $\mathbf{W}$
denote the flipping group of $S.$  By a {\it $\mathbf{W}$-submodule} of $F_2^n$ we mean a subspace
$U$ of $F_2^n$ such that $\mathbf{W}U\subseteq U.$

\begin{cor}\label{c3.3}
The subspace $U$ spanned by the vectors in $\Pi$ is a $\mathbf{W}$-submodule of $F_2^n.$
\end{cor}
\begin{proof}
From Lemma~\ref{l4.2}, $U$ is closed under the action of $\mathbf{s_1}, \mathbf{s_2}, \ldots,
\mathbf{s_{n-1}}.$ Note that for $\overline{i}\in \Pi$ we have \begin{eqnarray*} \mathbf{s_n}
\overline{i}&=&\left\{
  \begin{array}{ll}
   \overline{i}, & \hbox{if  $\overline{i}\in \Pi_0$;} \\
  \overline{i}+\displaystyle\sum\limits_{\overline{j}\in
\Pi_0} \overline{j}, & \hbox{if $\overline{i}\in \Pi_1$}
  \end{array}
\right.\\
&\in& U
\end{eqnarray*}
by Lemma~\ref{l4.1}.
\end{proof}

\begin{prop}\label{p3.4} The subspace $U$ in Corollary~\ref{c3.3} has the basis
$$\left\{
      \begin{array}{ll}
        \Pi, &\hbox{ if $|\Pi_1|$ is odd;} \\
        \Pi-\{\overline{j}\}, &\hbox{if $|\Pi_1|$ is even}
      \end{array}
    \right.$$
    for any $\overline{j}\in\Pi.$
Moreover $\widetilde{s}_n\not\in U$ if $|\Pi_1|$ is even.
\end{prop}
\begin{proof}
By Lemma~\ref{l3.2}, $\overline{1}, \overline{2}, \ldots, \overline{n-1}$ are linearly independent
and hence $U$ has dimension at least $n-1.$ Since $\widetilde{s}_n\not\in {\rm Span}\{\overline{1},
\overline{2}, \ldots, \overline{n-1}\},$ the proposition follows from the second case of
Lemma~\ref{l3.2}.
\end{proof}

Let $\mathbf{W}_{P}$ denote the subgroup of $\mathbf{W}$ generated by $\mathbf{s_1},$
$\mathbf{s_2},$ $\ldots,$ $\mathbf{s_{n-1}}.$ From Lemma~\ref{l4.2}, Proposition~\ref{p3.4} and the
fact $G\widetilde{s}_n=\widetilde{s}_n$ for $G\in\mathbf{W}_{P}$,
 we have the following corollary.

\begin{cor}\label{c3.5}
The subgroup $\mathbf{W}_{P}$ of $\mathbf{W}$  is isomorphic to the symmetric group $S_n$ on
$\Pi.$\hfill $\Box$
\end{cor}

Let $S'$ be another graph satisfying Assumption~\ref{ass1}, $\mathbf{s'_n}$ be the corresponding
matrix in Definition~\ref{d2.1} and $\Pi', \Pi'_0, \Pi'_1$ be the corresponding sets of vectors in
(\ref{ed1})-(\ref{ed3}). For this moment we suppose $|\Pi_1|=|\Pi'_1|.$ Let
$f:\Pi\cup\{\widetilde{s}_n\}\rightarrow \Pi'\cup\{\widetilde{s}~'_n\}$ be a bijection such that
$f(\widetilde{s}_n)=\widetilde{s}~'_n$ and $f(\Pi_1)=\Pi'_1.$ Then
$$\mathbf{s'_n}f(\widetilde{s}_n)=f(\widetilde{s}_n)+\displaystyle\sum\limits_{\overline{j}\in
\Pi_0}f(\overline{j})$$
and
$$\mathbf{s'_n}f(\overline{i})=\left\{
                                   \begin{array}{ll}
                                     f(\overline{i}), & \hbox{if $\overline{i}\in \Pi_0$;} \\
                                     f(\overline{i})+\displaystyle\sum\limits_{\overline{j}\in
\Pi_0}f(\overline{j}), & \hbox{if $\overline{i}\not\in\Pi_0$}
                                   \end{array}
                                 \right.
$$
are corresponding to the way that $\mathbf{s_n}$ acts on $\Pi\cup\{\widetilde{s}_n\}.$
 From Corollary~\ref{c3.5} and the above arguments we have the following
theorem.

\begin{thm}\label{t3.9}
$\mathbf{W}$ is unique up to isomorphism among all the graphs satisfying Assumption~\ref{ass1} with
a given cardinality $|\Pi_1|$ computed from (\ref{p3.2}).~~~~~~~~ \hfill $\Box$
\end{thm}

The flipping group $\mathbf{W}$ of a simply-laced Dynkin diagram $S$ is isomorphic to the quotient
group $W/Z(W)$ of the Coxeter group $W$ of $S$ by its center $Z(W)$\cite{hw:pre}, and the study of
Coxeter groups $W$ is notoriously interesting. With this in mind, one might expect expect the
flipping groups are very different on different graphs. Theorem~\ref{t3.9} is surprising since up
to isomorphism the number of flipping groups is at most $n-1$, which is much less than the number
of graphs satisfying Assumption~\ref{ass1}.

\section{Simple basis $\Delta$ of $F_2^n$}\label{s4}

To better describe the orbits in $\mathcal{P}$ later,  we need to choose a new basis of $F_2^n.$
Set
$$\Delta:=\left\{
            \begin{array}{ll}
              \Pi, & \hbox{if $|\Pi_1|$ is odd;} \\
              \Pi\cup\{\overline{n+1}\}-\{\overline{n}\}, & \hbox{if $|\Pi_1|$ is even,}
            \end{array}
          \right.
$$
where $\overline{n+1}:=\widetilde{s}_n.$ With referring to Proposition~\ref{p3.4}, $\Delta$ is a
basis of $F_2^n.$ To distinguish from the {\it standard basis} $\{\widetilde{s}_1, \widetilde{s}_2,
\ldots, \widetilde{s}_n\}$ of $F_2^n$, we refer $\Delta$ to the {\it simple basis} of $F_2^n$. For
each vector $u\in F_2^n$, $u$ can be written as a linear combination of elements in $\Delta,$ so
let $\Delta(u)$ be the subset of $\Delta$ such that
$$u=\sum\limits_{\overline{i}\in\Delta(u)}\overline{i},$$
set $sw(u):=|\Delta(u)|$, and we refer $sw(u)$ to be the {\it simple weight} of $u$. Note that for
$1\leq i\leq n-1,$ the vector $\overline{1}+\overline{2}+\cdots+\overline{i}$ has simple weight
$i$, but has weight
\begin{equation}\label{en4.1}
w(\overline{1}+\overline{2}+\cdots+\overline{i})=\left\{
                                                   \begin{array}{ll}
                                                     1, & \hbox{if $|[\overline{i}]\cap \Pi_1 |$ is even;} \\
                                                     2, & \hbox{if $|[\overline{i}]\cap \Pi_1 |$ is odd}
                                                   \end{array}
                                                 \right.
\end{equation}
by Lemma~\ref{l3.2}.
\bigskip

The following notation will be used in the sequel. For $V\subseteq F_2^n$ and $T\subseteq \{0, 1,
\ldots, n\},$
$$V_T:=\{u\in V~|~sw(u)\in T\},$$ and for shortness $V_{t_1, t_2, \ldots, t_i}:=V_{\{t_1, t_2, \ldots,
t_i\}}.$ Let $odd$ be the subset of $\{1, 2, \ldots, n\}$ consisting of odd integers.
\bigskip

%Now, give an order for $\Delta$ as follows. Put all vectors of $\Delta$ in $\Pi_0$ first and then put all vectors of $\Delta$ in $\Pi_1$ secondly and put $\overline{n+1}$ last if $\overline{n+1}\in \Delta.$ Let $P$ denote the change-of-basis matrix from the standard basis $\{\widetilde{s}_1,\widetilde{s}_2,\ldots,\widetilde{s}_n\}$ to the simple basis $\Delta.$ Due to Lemma~\ref{l3.2}, Lemma~\ref{l4.1} and Corollary~\ref{c3.5}, the following theorem holds.

%\begin{thm}\label{upto|Pi_1|}
%Let the order simple basis $\Delta$ and the matrix $P$ as above. Then $P\mathbf{W}P^{-1}$ is unique up to $|\Pi_1|.$ In particular, $\mathcal{P}$ is conjugate up to $|\Pi_1|.$
%\end{thm}

%From Theorem~\ref{upto|Pi_1|}, we determine $\mathcal{P}$ only depend on $|\Pi_1|$ in Section 5,6.

\section{The case $|\Pi_1|$ is odd}\label{s5}

In this section we assume $|\Pi_1|$ to be odd and the counter part
is treated in the next section. Note that $\Delta=\{\overline{1},
\overline{2}, \ldots, \overline{n}\}$ is a basis of $U=F_2^n$ in
this case. %Note that
%$$
%\overline{n-i+1}+\overline{n-i+2}+\cdots+\overline{n}=\left\{
%                                                        \begin{array}{ll}
%                                                          \widetilde{s}_{n-i}, & \hbox{if $|([\overline{n}]-[\overline{n-i}])\cap \Pi_1|$ is even ;} \\
%                                                          \widetilde{s}_{n-i}+\widetilde{s}_n, & \hbox{else.}
%                                                        \end{array}
%                                                      \right.
%$$
%Hence  a similar equation to (\ref{en4.1}) is
%\begin{equation}\label{enn5.0}
%w(\overline{n-i+1}+\overline{n-i+2}+\cdots+\overline{n})=\left\{
%                                                        \begin{array}{ll}
%                                                          1, & \hbox{if $([\overline{n}]-[\overline{n-i}])\cap \Pi_1|$) is even ;} \\
%                                                          2, & \hbox{else.}
%                                                        \end{array}
%                                                      \right.
%\end{equation}
%No other vectors with simple weight $i$ except the possible one in (\ref{en4.1}) and the possible
%one in (\ref{enn5.0}) can have weight $1.$
% Set $$I:=\{i~|~\overline{i}\in\Delta, |[\overline{i}]\cap \Pi_1 |{\rm ~is~
%even~or}~|([\overline{n}]-[\overline{n-i}])\cap \Pi_1|){\rm ~ is~ even}\},$$ and note that
From Lemma~\ref{l3.2}, for $1\leq i\leq n-1,$
$$
\widetilde{s}_{i}=\left\{
                                                   \begin{array}{ll}
                                                     \overline{1}+\overline{2}+\cdots+\overline{i}, & \hbox{if $|[\overline{i}]\cap \Pi_1|$ is even;} \\
                                                     \overline{i+1}+\overline{i+2}+\cdots+\overline{n}, & \hbox{if
$|[\overline{i}]\cap \Pi_1|$ is odd,}
                                                    \end{array}
                                              \right.
$$
and
$$
\widetilde{s}_{n}=\overline{1}+\overline{2}+\cdots+\overline{n}.
$$
 Hence, for $1\leq i\leq n-1,$
$$
sw(\widetilde{s}_{i})=\left\{
                                                   \begin{array}{ll}
                                                     i, & \hbox{if $|[\overline{i}]\cap \Pi_1|$ is even;} \\
                                                     n-i, & \hbox{if $|[\overline{i}]\cap \Pi_1|$ is odd,}
                                                    \end{array}
                                              \right.
$$
and $sw(\widetilde{s}_{n})=n.$ In other words, there exists a vector with simple weight $i$ and
weight $1$ if and only if $|[\overline{i}]\cap \Pi_1|$ is even, $i=n$ or $|[\overline{n-i}]\cap
\Pi_1|$ is odd. Set
$$I:=\{i\in [n]~|~ |[\overline{i}]\cap \Pi_1 |{\rm ~is~
even,~}i=n~{\rm or~}|[\overline{n-i}]\cap \Pi_1|{\rm ~ is~ odd}\},$$ where $[n]:=\{1, 2, \ldots,
n\}.$ Note that $w(U_i)\leq 2$ by Lemma~\ref{l3.2}, and
\begin{equation}\label{en5.0}
w(U_i)=1\qquad {\rm if~and~ only~ if~}\qquad i\in I
\end{equation}
 for $1\leq i\leq n$.

\begin{lem}\label{ln5.1} For $u\in F_2^n,$ %such that $|\Delta(u)\cap \Pi_1|\not=\emptyset,$
we have
$$\mathbf{s_n}u=\left\{
                  \begin{array}{ll}
                    u, & \hbox{if $|\Delta(u)\cap \Pi_1|$ is even;} \\
                    u+\sum\limits_{\overline{i}\in\Pi_0}\overline{i}, & \hbox{else.}
                  \end{array}
                \right.
$$ In particular,
$$sw(\mathbf{s_n}u)=\left\{
                  \begin{array}{ll}
                    sw(u), & \hbox{if $|\Delta(u)\cap \Pi_1|$ is even;} \\
                    n-|\Pi_1|+2k-sw(u), & \hbox{else,}
                  \end{array}
                \right.$$
where $k=|\Pi_1\cap\Delta(u)|.$
\end{lem}
\begin{proof} If $|\Delta(u)\cap \Pi_1|$ is even then $<u, \widetilde{s}_n>=0$ and
$\mathbf{s_n}u=u$ by construction. If $|\Delta(u)\cap \Pi_1|$ is odd, then
\begin{eqnarray*}
\mathbf{s_n}u&=&u+\sum\limits_{k=1}^m\widetilde{s}_{j_k}\\
             &=&u+\sum\limits_{\overline{i}\in\Pi_0}\overline{i}
\end{eqnarray*}
by Lemma~\ref{l4.1}, and $sw(\mathbf{s_n}u)=|\Delta(u)\cap \Pi_1|+(|\Pi_0|-|\Delta(u)\cap \Pi_0|)=
n-|\Pi_1|+2k-sw(u).$
\end{proof}

 The following lemma follows from Corollary~\ref{c3.5} and $\Delta=\Pi.$

\begin{lem}\label{l2.1}
The nontrivial orbits of $F_2^n$ under $\mathbf{W}_{P}$ are $U_i$ for $1\leq i\leq
n.$~~~~~~~~~~~~~~ \hfill $\Box$
\end{lem}

The following theorem solves the flipping puzzle when  $3\leq |\Pi_1|\leq n-3.$

\begin{thm}\label{odd_1}
Suppose $3\leq |\Pi_1|\leq n-3.$ Then the nontrivial orbits of $F_2^n$ under $\mathbf{W}$ are
$U_{A_1}, U_{A_2}, U_{A_3}, U_{A_4},$ where $$A_i:=\{j\in[n]~|~j\equiv i,n+|\Pi_1|-i\pmod{4}\}.$$
In particular the number of orbits (including the trivial one) of $F_2^n$ under $\mathbf{W}$ is
$$|\mathcal{P}|=\left\{
      \begin{array}{ll}
        3, & \hbox{if $n$ is even;} \\
        4, & \hbox{else,}
      \end{array}
    \right.$$
and the maximum-orbit-weight $M(S)$ of $S$ is
$$M(S)=\left\{
         \begin{array}{ll}
           1, & \hbox{if $A_i\cap I\not=\emptyset$ for all $i$;} \\
           2, & \hbox{else.}
         \end{array}
       \right.
 $$
\end{thm}
\begin{proof} Fix an integer $1\leq i\leq n.$
By Lemma~\ref{l2.1}, $U_i$ is contained in an orbit of $F_2^n$ under $\mathbf{W}.$ To put two
orbits under $\mathbf{W}_P$ to an orbit under $\mathbf{W}$ is only by the action of $\mathbf{s_n}.$
Hence $U_i$ and $U_{n-|\Pi_1|+2k-i}$ are in the same orbit by Lemma~\ref{ln5.1}, where $k$ runs
through possible odd integers $|\Pi_1\cap\Delta(u)|$ for $u\in U_i.$ In fact $k$ is any odd number
that satisfies $k\leq |\Pi_1|$ and  $0\leq i-k\leq |\Pi_0|;$ equivalently
\begin{equation}\label{en5.2}
{\rm max}\{1, i+|\Pi_1|-n\}\leq k\leq {\rm min}\{|\Pi_1|, i\}.
\end{equation}
Such an odd integer $k$ exists for any $1\leq i\leq n,$ and note that $$ n-|\Pi_1|+2k-i\equiv
n+|\Pi_1|-i\pmod 4$$ since $k$ and $|\Pi_1|$  are odd integers. To see the orbits as stated in the
theorem, it remains to show that $U_i$ and $U_{i+4}$ are in the same orbit under $\mathbf{W}$ for
$1\leq i\leq n-4.$ Set $k$ to be the least odd integer greater than or equal to ${\rm max}\{1,
i+|\Pi_1|-n+2\}.$
 For this $k$, (\ref{en5.2}) holds and then $U_i$ and
$U_{n-|\Pi_1|+2k-i}$ are in the same orbit. Here we use the assumption $|\Pi_1|\leq n-3$ to
guarantee the existence of  such $k$. Note that if we use $(n-|\Pi_1|+2k-i, k+2)$ to replace $(i,
k)$ in (\ref{en5.2}), we have
\begin{equation}\label{en5.3}{\rm max}\{1, 2k-i\}\leq k+2\leq {\rm min}\{|\Pi_1|, n-|\Pi_1|+2k-i\}.
\end{equation}
The above $k$ and the assumption $3\leq |\Pi_1|$ guarantee the equation (\ref{en5.3}). Since
$n-|\Pi_1|+2(k+2)-(n-|\Pi_1|+2k-i)=i+4,$ we have
 $U_{n-|\Pi_1|+2k-i}$ and $U_{i+4}$ in the same orbit. Putting these together, $U_i$ and $U_{i+4}$
are in the same orbit. The remaining statements of the theorem are obtained from the orbits
description immediately and by using (\ref{en5.0}).
\end{proof}

The following theorem does the remaining cases.

\begin{thm}\label{p5.4}
Suppose $|\Pi_1|=1,$ $n-2$ or $n-1.$ Then the nontrivial orbits of $F_2^n$ under $\mathbf{W}$ are
$$\left\{
    \begin{array}{ll}
      U_{i,n+1-i}, & \hbox{if $|\Pi_1|=1$;} \\
      U_{odd}, U_{2j}, & \hbox{if $|\Pi_1|=n-2$;}\\
    U_{2i-1, 2i}, & \hbox{if $|\Pi_1|=n-1$}\\
    \end{array}
  \right.
$$
for $1\leq i\leq \lceil n/2 \rceil$ and $1\leq j \leq (n-1)/2.$ In particular the number of orbits
(including the trivial one) of $F_2^n$ under $\mathbf{W}$ is
$$|\mathcal{P}|=\left\{
      \begin{array}{ll}
        \lceil (n+2)/2\rceil, & \hbox{if $|\Pi_1|=1$;} \\
       (n+3)/2, & \hbox{if $|\Pi_1|=n-2$;}\\
  (n+2)/2,               & \hbox{if $|\Pi_1|=n-1$,}
      \end{array}
    \right.$$
and the maximum-orbit-weight $M(S)$ of $S$ is at most $2$. Moreover  $M(S)=1$  if and only if
$$\left\{
    \begin{array}{ll}
      \{i, n+1-i\} \cap I\not=\emptyset\quad \hbox{\rm for all } 1\leq i\leq \lceil n/2\rceil, & \hbox{if $|\Pi_1|=1$;} \\
      odd\cap I\not=\emptyset~\hbox{\rm and}~U_{2j}\cap I\not=\emptyset\quad \hbox{\rm for all } 1\leq j\leq \lfloor n/2\rfloor, & \hbox {if $|\Pi_1|=n-2$;}\\
      \{2i-1, 2i\}\cap I\not=\emptyset\quad\hbox{\rm for all } 1\leq i\leq \lceil n/2\rceil, & \hbox {if $|\Pi_1|=n-1$.}\\
    \end{array}
  \right.
$$
\end{thm}

\begin{proof}
As the proof in Theorem~\ref{odd_1}, $U_i$ and $U_{n-|\Pi_1|+2k-i}$ are in the same orbit under
$\mathbf{W},$ where $k$ needs to satisfy (\ref{en5.2}). In the case $|\Pi_1|=1$, $k=1$ is the only
possible choice and hence $U_{n+1-i}$ is the only orbit under $\mathbf{W}_{P}$  been put together
with $U_i$ to become an orbit under $\mathbf{W}.$ In the case $|\Pi_1|=n-2,$ we have $k=i-2$ or $i$
if $i$ is odd; $k=i-1$ if $i$ is even. In the case $|\Pi_1|=n-1,$ we have $k=i$ if $i$ is odd;
$k=i-1$ if $i$ is even. In each of the remaining the proof follows similarly.
\end{proof}

\begin{exam}
Let $S$ be an odd cycle of length $n$, i.e. $n$ is odd,  $m=2$, $j_1=1$ and $j_2=n-1$. Then
$\Pi_0=\{\overline{1}, \overline{n}\}$ and $\Pi_1=\{\overline{2}, \overline{3},
\ldots,\overline{n-1}\}.$ Note that $|\Pi_1|=n-2$ is odd, and $I=\{1, 3, \ldots, n\}.$ Hence
Theorem~\ref{p5.4} applies. We have $$\mathcal{P}=\{U_{odd}, U_0, U_2, U_4, \ldots, U_{n-1}\}.$$ In
particular, $|\mathcal{P}|=(n+3)/2,$ and $M(S)=2.$
\end{exam}

\section{The case $|\Pi_1|$ is even}\label{s6}

In this section we assume $|\Pi_1|$ to be even.
 Recall that in this case $\Delta=\Pi\cup\{\overline{n+1}\}-\{\overline{n}\}$ and $\Delta-\{\overline{n+1}\}$ are bases of $F_2^n$ and $U$ respectively. Recall that
\begin{align}\label{e5.1}
\overline{1}+\overline{2}+\cdots+\overline{n}=0
\end{align}
Let $\overline{U}:=F_2^n-U,$ and note that
$\overline{U}=\overline{n+1}+U,$
$\overline{U}_1=\{\overline{n+1}\}$ and $U_{n}=\emptyset.$
From Lemma~\ref{l3.2}, for $1\leq i\leq n-1,$
$$
\widetilde{s}_{i}=\left\{
                                                   \begin{array}{ll}
                                                     \overline{1}+\overline{2}+\cdots+\overline{i}\in U, & \hbox{if $|[\overline{i}]\cap\Pi_1|$ is even;} \\
                                                     \overline{1}+\overline{2}+\cdots+\overline{i}+\overline{n+1}\in\overline{U}, & \hbox{if $|[\overline{i}]\cap\Pi_1|$ is odd,}
                                                    \end{array}
                                              \right.
$$
and
$$
\widetilde{s}_{n}=\overline{n+1}\in\overline{U}.
$$
Moreover, for $1\leq i\leq n-1,$
$$
sw(\widetilde{s}_{i})=\left\{
                                                   \begin{array}{ll}
                                                     i, & \hbox{if $|[\overline{i}]\cap \Pi_1|$ is even;} \\
                                                     i+1, & \hbox{if $|[\overline{i}]\cap \Pi_1|$ is odd,}\\
                                                    \end{array}
                                              \right.
$$
and $sw(\widetilde{s}_{n})=1.$ In other words, there exists a vector in $U$ with simple weight $i$
and weight $1$ if and only if $|[\overline{i}]\cap \Pi_1|$ is even; there exists a  vector in
$\overline{U}$ with simple weight $i$ and weight $1$ if and only if $|[\overline{i-1}]\cap \Pi_1|$
is odd or $i=1.$
%Note that $\overline{1}+\overline{2}+\cdots+\overline{i}\in U$ and
%$\overline{1}+\overline{2}+\cdots+\overline{i-1}+\overline{n+1}\in
%\overline{U}$ are the only two possible vectors with simple weight
%$i$ to have weight $1$ for $1\leq i\leq n.$
Set
$$
I=\{i\in[n-1]~|~ |[\overline{i}]\cap\Pi_1|{\rm~is~even}\}
$$
and
$$
J=\{i\in[n]~|~ |[\overline{i-1}]\cap\Pi_1|{\rm~is~odd~or~}i=1\}.
$$
Note that $w(U_i), w(\overline{U}_j)\leq 2$, and
\begin{equation}
\begin{split}
w(U_i)=1\qquad {\rm if~and~ only~ if~}\qquad i\in I;\\
w(\overline{U}_j)=1\qquad {\rm if~and~ only~ if~}\qquad j\in J
\end{split}
\end{equation}
 for $1\leq i\leq n-1$, $1\leq j\leq n.$

\begin{lem}\label{ln6.1}
For $u\in F_2^n$,  let $k=|\Pi_1\cap\Delta(u)|.$ Then the following (i),(ii) hold
\begin{enumerate}
\item[(i)] For $u\in U,$ we have
$$\mathbf{s_n}u=\left\{
                  \begin{array}{ll}
                    u, & \hbox{if $|\Delta(u)\cap \Pi_1|$ is even;} \\
                    u+\sum\limits_{\overline{i}\in\Pi_0}\overline{i}, & \hbox{else.}
                  \end{array}
                \right.
$$ In particular, the simple weight $sw(\mathbf{s_n}u)$ of $\mathbf{s_n}u$ is
$$\left\{
                  \begin{array}{ll}
                    sw(u), & \hbox{if $|\Delta(u)\cap \Pi_1|$ is even;} \\
                    n-|\Pi_1|+2k-sw(u), & \hbox{if  $|\Delta(u)\cap \Pi_1|$ is odd and $\overline{n}\in
\Pi_1;$}\\
                    sw(u)+|\Pi_1|-2k, & \hbox{else.}\\
                  \end{array}
                \right.$$
\item[(ii)] For $u\in\overline{U},$ we have
$$\mathbf{s_n}u=\left\{
                  \begin{array}{ll}
                    u, & \hbox{if $|\Delta(u)\cap \Pi_1|$ is odd;} \\
                    u+\sum\limits_{\overline{i}\in\Pi_0}\overline{i}, & \hbox{else.}
                  \end{array}
                \right.
$$ In particular, the simple weight $sw(\mathbf{s_n}u)$ of $\mathbf{s_n}u$ is
$$\left\{
                  \begin{array}{ll}
                    sw(u), & \hbox{if $|\Delta(u)\cap \Pi_1|$ is odd;} \\
                    n-|\Pi_1|+2k+2-sw(u), & \hbox{if  $|\Delta(u)\cap \Pi_1|$ is even and $\overline{n}\in
\Pi_1;$}\\
                    sw(u)+|\Pi_1|-2k, & \hbox{else.}\\
                  \end{array}
                \right.$$
\end{enumerate}
\end{lem}
\begin{proof}
The proof is similar to the proof of Lemma~\ref{ln5.1}, except that at this time since the choice
of simple basis $\Delta$ is different, the action of $\mathbf{s_n}$ on a vector is a little
different, and we need to use (\ref{e5.1}) to adjust the simple weight of a vector.
\end{proof}

By Corollary~\ref{c3.3} the orbits of $F_2^n$ under $\mathbf{W}$ (resp. under  $\mathbf{W}_{P}$)
are divided into two parts, one in
$U$ and the other in $\overline{U}.$ %First, we show the orbits in $U$ under $\mathbf{W}$ as follows.

\begin{lem}\label{l6.1}
The nontrivial orbits of $F_2^n$ under $\mathbf{W}_{P}$ are $\overline{U}_1$,
$\overline{U}_{i+1,n+1-i}$ and $U_{i,n-i}$ for $1\leq i\leq \lfloor n/2\rfloor$ \hfill $\Box$
\end{lem}

\begin{proof}
By construction, $\overline{U}_1=\{\widetilde{s}_n\}$ is an orbit under $\mathbf{W}_{P}.$ By
Corollary~\ref{c3.3} and Corollary~\ref{c3.5}, $U_i$ is contained in an orbit of $F_2^n$ under
$\mathbf{W}_{P}$ and $\overline{U}_i$ is contained in another one for $1\le i\leq n-1$. The
equation (\ref{e5.1}) and our choice of $\Delta$ imply that $U_i$ and $U_{n-i}$ are in the same
orbit of $F_2^n$ under $\mathbf{W}_{P}$; $\overline{U}_{i+1}$ and $\overline{U}_{n+1-i}$ are in
another one for $1\leq i\leq n-1$. Since no other ways to put these sets together, we have the
lemma.
\end{proof}

%The following lemma is similar to Lemma~\ref{l2.2}.

%\begin{lem}\label{l6.2}
%Suppose $4\leq|\Pi_1|\leq n-3$. The following (i),(ii) hold
%\begin{enumerate}
%\item[(i)]if $\overline{n}\in \Pi_1,$ then $U_{k,n-|\Pi_1|+2-k}$ is contained in some orbit of $U$ under $\mathbf{W}$ and $\overline{U}_{k,n-|\Pi_1|+2-k}$ is contained in some orbit of $\overline{U}$ under $\mathbf{W}$ for $1\leq k\leq 4;$
%\item[(ii)]if $\overline{n}\notin \Pi_1,$ then $U_{n-k,n-|\Pi_1|+2-k}$ is contained in some orbit of $U$ under $\mathbf{W}$ and $\overline{U}_{n-k,n-|\Pi_1|+2-k}$ is contained in some orbit of $\overline{U}$ under $\mathbf{W}$ for $1\leq k\leq 4.$\hfill $\Box$
%\end{enumerate}
%\end{lem}

\begin{thm}\label{even_1}
Suppose $4\leq |\Pi_1|\leq n-3.$  Then the nontrivial orbits of $F_2^n$ under $\mathbf{W}$ are
$U_{B_1}, U_{B_2}, U_{B_3}, U_{B_4},
\overline{U}_{C_1},\overline{U}_{C_2},\overline{U}_{C_3},\overline{U}_{C_4},$ where
$$B_i=\{j\in[n-1]~|~j\equiv i, i+|\Pi_1|-2, n-i, n-i+|\Pi_1|-2\pmod{4}\}$$ and $$C_i=\{j\in[n]~|~j\equiv i, i+|\Pi_1|, n+2-i, n+2-i+|\Pi_1| \pmod{4}\}.$$ In
particular the number of orbits (including the trivial one) of $F_2^n$ under $\mathbf{W}$ is
$$|\mathcal{P}|=\left\{
      \begin{array}{ll}
        6, & \hbox{if $n$ is even;} \\
        4, & \hbox{else,}
      \end{array}
    \right.$$
and the maximum-orbit weight $M(S)$ of $S$ is
$$M(S)=\left\{
         \begin{array}{ll}
           1, & \hbox{if $B_i\cap I\not=\emptyset$ and $C_i\cap J\not=\emptyset$ for all $i$;} \\
           2, & \hbox{else.}
         \end{array}
       \right.
 $$
\end{thm}
\begin{proof} Firstly we determine the orbits of $U$ under $\mathbf{W}$. By Lemma~\ref{l6.1}, $U_{i, n-i}$ is contained in an orbit under $\mathbf{W}$ for $1\leq i\leq
n-1.$ We suppose $\overline{n}\in \Pi_0$ and the case $\overline{n}\in \Pi_1$ is left to the
reader. In this case $U_i$ and $U_{i+|\Pi_1|-2k}$ are in the same orbit of $F_2^n$ under
$\mathbf{W}$ by Lemma~\ref{ln6.1}(i), where $1\leq i+|\Pi_1|-2k\leq n-1$ and $k$ runs through
possible odd integers $|\Pi_1\cap \Delta(u)|$ for $u\in U_i.$ In fact $k$ is any odd number that
satisfies $k\leq |\Pi_1|-1$ and $0\leq i-k\leq |\Pi_0|-1;$ equivalently
\begin{equation}\label{en6.2}
{\rm max}\{1, i+|\Pi_1|-n+1\}\leq k\leq {\rm min}\{|\Pi_1|-1, i\}.
\end{equation}
Such an odd $k$ exists for any $1\leq i\leq n-3,$ and note that $$ i+|\Pi_1|-2k\equiv
i+|\Pi_1|-2\pmod 4.$$  To determine the orbits of $U$ under $\mathbf{W}$ in this case, it remains
to show that $U_i$ and $U_{i+4}$ are in the same orbit under $\mathbf{W}$ for $1\leq i\leq \lfloor
n/2 \rfloor.$ Suppose $4\leq |\Pi_1|\leq 6.$ Set $k=1$ to conclude $U_i$ and $U_{i+2}$ in an orbit
if $|\Pi_1|=4$; $U_i$ and $U_{i+4}$ in an orbit if $|\Pi_1|=6.$ Suppose $|\Pi_1|\geq 8.$ Then
$n\geq 11$ and $\lfloor n/2\rfloor\leq n-6.$ Set $k$ to be the least odd integer greater than or
equal to ${\rm max}\{1, i+|\Pi_1|-n+3\}.$ For this $k$, (\ref{en6.2}) holds and then $U_i$ and
$U_{i+|\Pi_1|-2k}$ are in the same orbit. Here we use the assumption $|\Pi_1|\leq n-3.$ Note that
if we use $(i+|\Pi_1|-2k, |\Pi_1|-k-2)$ to replace $(i, k)$ in (\ref{en6.2}), we have
\begin{equation}\label{en6.3}{\rm max}\{1, i+2|\Pi_1|-2k-n+1\}\leq |\Pi_1|-k-2\leq {\rm
min}\{|\Pi_1|-1, i+|\Pi_1|-2k\}.
\end{equation}
The above $k$, the assumption $4\leq |\Pi_1|$ and $i\leq n-6$ guarantee the equation (\ref{en6.3}).
Since $(i+|\Pi_1|-2k)+|\Pi_1|-2(|\Pi_1|-k-2)=i+4,$ we have
 $U_{i+|\Pi_1|-2k}$ and $U_{i+4}$ in the same orbit. Putting these together, $U_i$ and $U_{i+4}$
are in the same orbit. Then the orbits of $U$ under $\mathbf{W}$ are $U_{B_1}$, $U_{B_2}$,
$U_{B_3},$ $U_{B_4}$ as in the statement.
\medskip

Secondly, we determine the orbits of $\overline{U}$ under $\mathbf{W}.$ Since the proof is similar
to the above case, we only give a sketch. By Lemma~\ref{l6.1}, $\overline{U}_{i, n+2-i}$ is
contained in an orbit for $2\leq i\leq n.$ We suppose $\overline{n}\in \Pi_1$ and leave the case
$\overline{n}\in \Pi_0$ to the reader. By Lemma~\ref{ln6.1}(ii), we have $U_i$ and
$U_{n-|\Pi_1|+2k+2-i}$ in an orbit, where $k=|\Delta(u)\cap\Pi_1|$ is an even number for some $u\in
U_i$ and $1\leq i\leq n-4.$ From the same argument with $k$ been replaced by $k+2$, we find
$U_{n-|\Pi_1|+2k+2-i}$ and $U_{i+4}$ in an orbit to finish the proof.
\medskip

 The remaining statements of the theorem are obtained from the orbits description.
\end{proof}

The following theorem determine the nontrivial orbits of $F_2^n$ under $\mathbf{W}$ in the
remaining cases.

\begin{thm}\label{p6.4}
Suppose $|\Pi_1|=2,$ $n-2$ or $n-1.$ Then with referring to the notation in Theorem~\ref{even_1},
the nontrivial orbits of $F_2^n$ under $\mathbf{W}$ are
$$
\left\{
\begin{array}{ll}
U_{i,n-i},\overline{U}_{C_1}, \overline{U}_{C_2},      & \hbox{if $|\Pi_1|=2;$}\\
U_{odd},U_{2j,n-2j},\overline{U}_{odd},\overline{U}_{2t,n+2-2t}, & \hbox{if $|\Pi_1|=n-2;$}\\
U_{2j-1,2j,n-2j,n+1-2j},\overline{U}_{2t-1,2t,n+2-2t,n+3-2t},  & \hbox {if $|\Pi_1|=n-1,$}
\end{array}
\right.
$$
for $1\leq i\leq \lfloor n/2\rfloor,$ $1\leq j\leq \lceil (n-2)/4\rceil$ and $1\leq t\leq \lceil
n/4\rceil.$ In particular the number of orbits (including the trivial one) of $F_2^n$ under $W$ is
$$|\mathcal{P}|=
\left\{
      \begin{array}{ll}
        (n+6)/2, & \hbox{if $|\Pi_1|=2$ and $n$ is even, or $|\Pi_1|=n-2;$} \\
        (n+3)/2, & \hbox{if $|\Pi_1|=2$ and $n$ is odd, or $|\Pi_1|=n-1,$}
      \end{array}
    \right.$$
and the maximum-orbit-weight $M(S)$ of $S$ is at most $2$. Moreover $M(S)=1$ if and only if
$$\left.
    \begin{array}{ll}
      ~~~~\{i, n-i\}\cap I\not=\emptyset~{\rm and}~\overline{U}_{C_j}\cap J\not=\emptyset~{\rm for~}1\leq j\leq 2, & \hbox{if $|\Pi_1|=2$;} \\    \\
\left\{
  \begin{array}{ll}
  odd\cap I\not=\emptyset, \{2j, n-2j\}\cap I\not=\emptyset \\
{\rm for~all~} 1\leq j\leq \lceil (n-2)/4\rceil,       \\
 odd\cap J\not=\emptyset, \{2t, n+2-2t\}\cap J\not=\emptyset \\
{\rm for~all~} 1\leq t\leq \lceil n/4\rceil,
  \end{array}
\right.
& \hbox{if $|\Pi_1|=n-2$;}      \\       \\
\left\{
  \begin{array}{ll}
\{2j-1, 2j, n-2j, n+1-2j\}\cap I\not=\emptyset\\
 {\rm for~ all~} 1\leq j\leq \lceil
(n-2)/4\rceil,\\
\{2t-1, 2t, n+2-2t, n+3-2t\}\cap J\not=\emptyset\quad\\
 {\rm for~ all~} 1\leq t\leq \lceil n/4\rceil,
  \end{array}
\right.    &\hbox{ if $|\Pi_1|=n-1.$}
    \end{array}
  \right.
 $$
%$$\left\{
%    \begin{array}{ll}
%    U_{k,n-k},  & \hbox {if $|\Pi_1|=2$;}\\
    %U_{2\ell-1,2\ell,n-2\ell,n+1-2\ell},  & \hbox {if $|\Pi_1|=n-1$;}\\
%    U_{odd},U_{2\ell,n-2\ell},  & \hbox {if $|\Pi_1|=n-2$}
%    \end{array}
%  \right.
%$$
%for $1\leq k\leq \lfloor\frac{n}{2}\rfloor$ and $1\leq \ell\leq \lceil\frac{n-2}{4}\rceil.$
\end{thm}
\begin{proof}
The proof is similar to the proof of Theorem~\ref{p5.4} that follows from the proof of
Theorem~\ref{odd_1}. At this time, to determine the orbits of $U$ we check what values of odd $k$
occur in (\ref{en6.2}) in each case of $|\Pi_1|\in\{2, n-2, n-1\}$. To determine the orbits of
$\overline{U}$ under $\mathbf{W}$, we do similarly as in the second part of the proof of
Theorem~\ref{even_1}.
\end{proof}

\begin{exam}
Let $S$ be an even cycle of length $n$, i.e. $n$ is even,  $m=2$, $j_1=1$ and $j_2=n-1$. Then
$\Pi_0=\{\overline{1}, \overline{n}\}$ and $\Pi_1=\{\overline{2}, \overline{3},
\ldots,\overline{n-1}\}.$ Note that $|\Pi_1|=n-2$ is even and $I=J=\{1,3,\ldots,n-1\}.$ Hence
Theorem~\ref{p6.4} applies. We have
$$\mathcal{P}=\{U_{odd}, U_0, U_{2, n-2},  U_{4, n-4},\ldots
,U_{2j, n-2j},\overline{U}_{odd}, \overline{U}_{2, n}, \overline{U}_{4, n-2}, \ldots,
\overline{U}_{2t, n-2t+2}\},$$ where $j=\lceil(n-2)/4\rceil$ and $t=\lceil n/4\rceil.$
In particular
$$|\mathcal{P}|=\lceil (n-2)/4\rceil+\lceil n/4\rceil+3=(n+6)/2,$$
 and $M(S)=2.$
\end{exam}

\section{Summary}\label{s7}

We list the main results as follows. Let $S$ be a connected graph with $n$ vertices $s_1, s_2,
\ldots, s_n$ that contains an induced path $s_1, s_2, \ldots, s_{n-1}$ of $n-1$ vertices, and $s_n$
has neighbors $s_{j_1}, s_{j_2}, \ldots, s_{j_m}$ with $1\leq j_1<j_2\cdots<j_m\leq n-1.$ Let
$\widetilde{s}_1, \widetilde{s}_2, \ldots, \widetilde{s}_n$ denote the characteristic vectors of
$F_2^n$ and let $ \mathbf{s_1}, \mathbf{s_2}, \ldots, \mathbf{s_n}$ denote the flipping moves
associated with $s_1, s_2, \ldots, s_n$ respectively.
\bigskip

 Set
$$
\overline{1}=\widetilde{s}_1,~\overline{i+1}=\mathbf{s_i}\mathbf{s_{i-1}}\cdots
\mathbf{s_1}\overline{1}\quad (1\leq i\leq n-1), \quad \overline{n+1}:=\widetilde{s}_n.
$$ and consider the following three sets
\begin{eqnarray*}
\Pi_{~}&=&\{\overline{1}, \overline{2}, \ldots, \overline{n}\},\\
\Pi_0&=&\{\overline{i}\in\Pi~|~<\overline{i}, \widetilde{s}_n>=0\},\\
\Pi_1&=&\Pi-\Pi_0.
\end{eqnarray*}
By using the graph structure we can compute the following value
$$
|\Pi_1|=
      \sum^{\lceil\frac{m}{2}\rceil}_{k=1}j_{2k}-j_{2k-1}.
      $$
as shown in Proposition~\ref{p3.2}.  Let
$$\Delta:=\left\{
            \begin{array}{ll}
              \Pi, & \hbox{if $|\Pi_1|$ is odd;} \\
              \Pi\cup\{\overline{n+1}\}-\{\overline{n}\}, & \hbox{if $|\Pi_1|$ is even}
            \end{array}
          \right.
$$
be the simple basis of $F_2^n$ as shown in the beginning of Section~\ref{s4}. For a vector $u\in
F_2^n$ let $sw(u)$ denote the simple weight of $u$, i.e. the number nonzero terms in writing $u$ as
a linear combination of elements in $\Delta.$ Let $U$ be the subspace spanned by the vectors in
$\Pi.$ For $V\subseteq F_2^n$ and $T\subseteq \{0, 1, \ldots, n\},$
$$V_T:=\{u\in V~|~sw(u)\in T\},$$ and for shortness $V_{t_1, t_2, \ldots, t_i}:=V_{\{t_1, t_2, \ldots,
t_i\}}.$ Let $odd$ be the subset of $\{1, 2, \ldots, n\}$ consisting of odd integers. Set
\begin{eqnarray*}
A_i&=&\{j\in[n]~|~j\equiv i,n+|\Pi_1|-i\pmod{4}\},\\
B_i&=&\{j\in[n-1]~|~j\equiv i, i+|\Pi_1|-2, n-i, n-i+|\Pi_1|-2\pmod{4}\},\\
C_i&=&\{j\in[n]~|~j\equiv i, i+|\Pi_1|, n+2-i, n+2-i+|\Pi_1| \pmod{4}\}.
\end{eqnarray*}
Let $\mathcal{P}$ denote the set of orbits of the flipping puzzle on $S.$ Then the set
$\mathcal{P}$ and its cardinality $|\mathcal{P}|$ are given in the following table according to the
different cases of the pair $(|\Pi_1|, n)$ in the first two columns.

\bigskip

\begin{center}
\begin{tabular}{cccc}
\hline
$|\Pi_1|$&$n$&
{\small $\left.
\begin{array}{c}
{\rm nontrivial~}O\in \mathcal{P}\\
\hbox{(might be repeated)}
\end{array}
\right.$} & $|\mathcal{P}|$
 \\
\hline \\[-3pt]
{\small $\left.
\begin{array}{c}
3\leq|\Pi_1|\leq n-3,\\
|\Pi_1| \hbox{~is odd}
\end{array}
\right.$}
  &even  &$U_{A_j}$ & $3$ \\
 \hline \\[-5pt]
{\small $\left.
\begin{array}{c}
3\leq|\Pi_1|\leq n-3,\\
|\Pi_1| \hbox{~is odd}
\end{array}
\right.$}
  &odd  &$U_{A_j}$ & $4$   \\
\hline \\[-5pt]
{\small $\left.
\begin{array}{c}
4\leq|\Pi_1|\leq n-3,
\\
|\Pi_1| \hbox{~is even}
\end{array}
\right.$}
  &even&$U_{B_j}, \overline{U}_{C_j}$ & $6$
            \\
\hline \\[-5pt]
{\small $\left.
\begin{array}{c}
4\leq|\Pi_1|\leq n-3,\\
|\Pi_1| \hbox{~is even}
\end{array}
\right.$}&odd&$U_{B_j}, \overline{U}_{C_j}$&$4$
     \\
\hline \\[-5pt]
$|\Pi_1|=1$
  &   &$U_{t, n+1-t}$&  $\lceil (n+2)/2\rceil$  \\
\hline \\[-5pt]
$|\Pi_1|=2$
  &even&$U_{i, n-i}, \overline{U}_{C_1}, \overline{U}_{C_2} $ & $\displaystyle (n+6)/2$
 \\
\hline \\[-5pt]
$|\Pi_1|=2$&odd&$U_{i, n-i},\overline{U}_{C_1}, \overline{U}_{C_2} $&$\displaystyle (n+3)/2$                                                                     \\
\hline \\[-5pt]
{\small $\left.
\begin{array}{c}
|\Pi_1|=n-2,\\
|\Pi_1| \hbox{~is odd}
\end{array}
\right.$}
  &odd  & $U_{odd}, U_{2i}$ & $\displaystyle (n+3)/2$
\\
 \hline \\[-5pt]
    {\small $\left.
\begin{array}{c}
|\Pi_1|=n-2,\\
|\Pi_1| \hbox{~is even}
\end{array}
\right.$}
  &even  & $\left.
          \begin{array}{ll}
           U_{odd}, U_{2h, n-2h}, \\
             \overline{U}_{odd}, \overline{U}_{2g, n+2-2g}
          \end{array}
        \right.
$ &$\displaystyle (n+6)/2$
\\
 \hline \\[-5pt]
{\small $\left.
\begin{array}{c}
|\Pi_1|=n-1,\\
|\Pi_1| \hbox{~is odd}
\end{array}
\right.$}
  &even  & $U_{2t-1, 2t}$ & $\displaystyle (n+2)/2$
\\
 \hline \\[-5pt]
    {\small $\left.
\begin{array}{c}
|\Pi_1|=n-1,\\
|\Pi_1| \hbox{~is even}
\end{array}
\right.$}
  & odd & $\left.
          \begin{array}{ll}
            U_{2h-1, 2h, n-2h, , n+1-2h}, \\
            \overline{U}_{2g-1, 2gn+2-2g, n+3-2g}
          \end{array}
        \right.
 $ & $\displaystyle (n+3)/2$
\\
 \hline \\[-5pt]
\end{tabular}

\bigskip

\noindent where $1\leq j\leq 4,$ $1\leq t\leq \lceil n/2\rceil,$ $1\leq i\leq \lfloor n/2\rfloor,$
$1\leq h\leq \lfloor (n-2)/4\rfloor,$ $1\leq g\leq \lceil n/4 \rceil.$

\bigskip
\bigskip

{\bf Tabel. 1}
\end{center}

%Secondly, we show the orbits in $\overline{U}$ under $\mathbf{W}.$

%\begin{prop}
%Suppose $|\Pi_1|\leq n-3.$ The following (i),(ii) hold.
%\begin{enumerate}
%\item[(i)] If $|\Pi_1|\equiv 0\pmod{4},$  then the orbits of $\overline{U}$ under $\mathbf{W}$ are $\overline{U}_{C_1}, \overline{U}_{C_2}, \overline{U}_{C_3}, \overline{U}_{C_4}$, where $C_k=\{i\in[n]~|~i\equiv k, n+2-k \pmod{4}\};$
%\item[(ii)] If $|\Pi_1|\equiv 2\pmod{4},$  then the orbits of $\overline{U}$ under $\mathbf{W}$ are $\overline{U}_{A_1}, \overline{U}_{A_2}$, where $A_k=\{i\in[n]~|~i\equiv k, n-k\pmod{2}\}.$
%\end{enumerate}
%\end{prop}

%Note that if $n$ is odd then $\overline{U}_{A_1}=\overline{U}_{A_2}=U$; if $n$ is even then
%$\overline{U}_{A_1}=\overline{U}_{odd}$ and $\overline{U}_{A_2}=\overline{U}_{even}.$
% The following proposition deals with the remaining cases of the orbits in $\overline{U}$ under
%$\mathbf{W}.$

%\begin{prop}
%The orbits of $\overline{U}$ under $\mathbf{W}$ are
%$$\left\{
%    \begin{array}{ll}
%    \overline{U}_{2k-1,2k,n+2-2k,n+3-2k},  & \hbox {if $|\Pi_1|=n-1$;}\\
%    \overline{U}_{odd},\overline{U}_{2k,n+2-2k},  & \hbox {if $|\Pi_1|=n-2$}
%    \end{array}
%  \right.
%$$
%for $1\leq k\leq \lceil\frac{n}{4}\rceil.$
%\end{prop}

\section{Appendix}\label{s8}
We are indebted to a referee for the information in this section. Let $S$ be a simple connected
graph with $n$ vertices and adjacency matrix $A$. The adjacency matrix defines an alternating form
 $<, >_A$ on $F_2^n$ by
$$<u, v>_A=u^tAv$$
and a quadratic form $q$ on $F_2^n$ that satisfies $q(\widetilde{s})=1$ and
$$q(u+v)=q(u)+q(v)+<u, v>_A$$ for all vertices $s\in S$ and $u, v\in F_2^n$. For a vertex $s\in S$, the
associating matrix $\mathbf{s}$ in Definition~\ref{d2.1} satisfies
\begin{equation}\label{e8.1} \mathbf{s}A\mathbf{s}^t=A.
\end{equation}
 Hence  $\mathbf{s}^t$ is an element of the symplectic group $S(n, F_2)$ \cite[p. 69]{t:92}, and therefore
the transpose group $\mathbf{W}^t$ of the flipping group $\mathbf{W}$ of $S$ is a subgroup of $S(n,
F_2).$ Moreover $\mathbf{W}^t$ preserves $q$ in the sense that $q(\mathbf{w}^tu)=q(u)$ for any
$\mathbf{w}^t\in \mathbf{W}^t$ and any $u\in F_2^n.$
 Note that from
 Definition~\ref{d2.1},
 \begin{equation}\mathbf{s}^tu=u+<\widetilde{s}, u>_A\widetilde{s}\end{equation}
for $s\in S$ and $u\in F_2^n.$  Such an  $\mathbf{s}^t$ is called a {\it transvection} in the
literature.  The study of arbitrary groups generated by transvections was largely instituted by
McLaughlin \cite{m:67, m:69}. Hamelink's work on Lie algebras led to a question about groups
generated by symplectic transvections over $F_2$ \cite{h:69}. Hamelink's question was answered by
Seidel, as reported and generalized by Shult in his Breukelen lectures \cite{s:73, s:74}. Graphical
notation is implicit in this earlier work and explicit in that of Brown and Humphries \cite{rh:85,
h:85}. A survey of related work, a brief discussion of Humphries results, and a discussion of the
isomorphism types of groups occurring are given by Hall \cite{h:84}. More recent results are in
\cite{r:05, s:05}.
\bigskip

Let $\mathcal{P}'$ denote the set of orbits under the action of $\mathbf{W}^t$ on $F_2^n$. Several
of the papers discussed above (or referenced therein) also focus on and discuss orbit lengths for
$\mathcal{P}'.$  As before let $\mathcal{P}$ be the set of orbits under the action of $\mathbf{W}$
on $F_2^n$ (the set of orbits of the flipping puzzle on $S$). By (\ref{e8.1}) and using
$\mathbf{s}^2=I$, the map
$$O\rightarrow AO$$
is a map from $\mathcal{P}'$ into $\mathcal{P},$ where $AO=\{Au~|~u\in O\}.$ In particular if $A$
is nonsingular over $F_2$, this map is a bijection. But when $A$ is singular, the orbits structures
can presumably differ. See more \cite{h:pre4} for more connections between $\mathcal{P}'$ and
$\mathcal{P}.$

\bigskip

\noindent Hau-wen Huang \hfil\break Department of Applied Mathematics \hfil\break National Chiao
Tung University \hfil\break 1001 Ta Hsueh Road \hfil\break Hsinchu, Taiwan 30050, R.O.C.\hfil\break
Email: {\tt poker80.am94g@nctu.edu.tw} \hfil\break Fax: +886-3-5724679 \hfil\break
\medskip

\noindent Chih-wen Weng \hfil\break Department of Applied Mathematics \hfil\break National Chiao
Tung University \hfil\break 1001 Ta Hsueh Road \hfil\break Hsinchu, Taiwan 30050, R.O.C.\hfil\break
Email: {\tt weng@math.nctu.edu.tw} \hfil\break Fax: +886-3-5724679 \hfil\break
\medskip

\end{document}